\documentclass{icmart}
\usepackage{amsfonts}
\usepackage{amsmath}

\usepackage{eurosym}
\usepackage{graphicx}
\usepackage{float}
\usepackage{subfigure}

\newtheorem{theorem}{Theorem}[section]

\theoremstyle{definition}

\newtheorem{corollary}[theorem]{Corollary}

\newtheorem{definition}[theorem]{Definition}
\newtheorem{example}[theorem]{Example}

\newtheorem{lemma}[theorem]{Lemma}

\newtheorem{problem}[theorem]{Problem}
\newtheorem{proposition}[theorem]{Proposition}
\newtheorem{remark}[theorem]{Remark}

\contact[terry.lyons@oxford-man.ox.ac.uk]{Oxford-Man Institute of Quantitative Finance,
  University of Oxford, England, OX2 6ED}

\title[Learning from Signatures]{Rough paths, Signatures and the modelling
of functions on streams}
\author[Terry Lyons]{Terry Lyons\thanks{%
Acknowledges the support of the Oxford-Man Institute, the support provided
by ERC advanced grant ESig (agreement no. 291244), and particularly the
contributions of his colleagues and his students without whom none of this
would have happened and in addition to Kelly Wyatt, Justin Sharp, Horatio Boedihardjo,  Hao Ni and
Danyu Yang for helping the author finalise this mss. The data analysis is is
reproduced from the cited paper with Gyurko et al., Gyurko did the analysis,
The raw data for that paper is available on Reuters.}}

\begin{document}

\maketitle

\begin{abstract}
Rough path theory is focused on capturing and making precise the
interactions between highly oscillatory and non-linear systems. The
techniques draw particularly on the analysis of LC Young and the geometric
algebra of KT Chen. The concepts and theorems, and the uniform estimates,
have found widespread application; the first applications gave simplified
proofs of basic questions from the large deviation theory and substantially
extending Ito's theory of SDEs; the recent applications contribute to
(Graham) automated recognition of Chinese handwriting and (Hairer)
formulation of appropriate SPDEs to model randomly evolving interfaces. At
the heart of the mathematics is the challenge of describing a smooth but
potentially highly oscillatory and vector valued path $x_{t}$ parsimoniously
so as to effectively predict the response of a nonlinear system such as $%
dy_{t}=f(y_{t})dx_{t}$, $y_{0}=a$. The Signature is a homomorphism from the
monoid of paths into the grouplike elements of a closed tensor algebra. It
provides a graduated summary of the path $x$. Hambly and Lyons have shown
that this non-commutative transform is faithful for paths of bounded
variation up to appropriate null modifications. Among paths of bounded
variation with given Signature there is always a unique shortest
representative. These graduated summaries or features of a path are at the
heart of the definition of a rough path; locally they remove the need to
look at the fine structure of the path. Taylor's theorem explains how any
smooth function can, locally, be expressed as a linear combination of
certain special functions (monomials based at that point). Coordinate
iterated integrals form a more subtle algebra of features that can describe
a stream or path in an analogous way; they allow a definition of rough path
and a natural linear "basis"\ for functions on streams that can be used for
machine learning.
\end{abstract}



\setcounter{MaxMatrixCols}{10}

\setcounter{tocdepth}{1} 
\tableofcontents
\pagebreak

\section{A path or a text?}

The mathematical concept of a path embraces the notion of an evolving or
time ordered sequence of events, parameterised by a continuous variable. Our
mathematical study of these objects does not encourage us to think broadly
about the truly enormous range of "paths" that occur. This talk will take an
analyst's perspective, we do not expect to study a particular path but
rather to find broad brush tools that allow us to study a wide variety of
paths - ranging form very "pure" mathematical objects that capture holonomy
to very concrete paths that describe financial data. Our goal will be to
explain the progress we have made in the last 50 years or so in describing
such paths effectively, and some of the consequences of these developments.

Let us start by noting that although most mathematicians would agree on a
definition of a path, most have a rather stereotyped and limited imagination
about the variety of paths that are "in the wild". One key observation is
that in most cases we are interested in paths because they represent some
evolution that interacts with and influences some wider system. Another is
that in most paths, in standard presentations, the content and influence are
locked into complex multidimensional oscillations.

\begin{figure}[H]
\centering
\includegraphics[
trim = 0mm 0mm 0mm 0mm, clip, width=0.80\textwidth]
{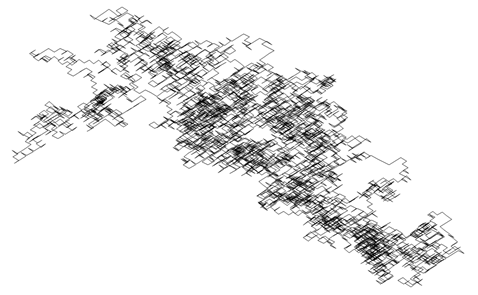}
\end{figure}

The path in the figure is a piece of text. Each character in the text is encoded using ascii as a byte of 8 bits,
each byte is represented as four letters of two bits, each two bit letter is represented by a line from the centre to one of the four corners of a square (for visial reasons the centre of this square is dispaced slightly to create a loop). The text can easily be represented in other ways, perhaps in different font or with each character as a bitmap.
Each stream has broadly the same effect on a coarse scale
although the detailed texture is perhaps a bit different.

\section{Financial Data or a Semimartingale}

One important source of sequential data comes from financial markets. An
intrinsic feature of financial markets is that they are high dimensional but
there is a strong notion of sequencing of events. Buying with future
knowledge is forbidden. Much of the information relates to prices, and one
of the radical successes of applied mathematics over the last 20-30 years
came out of the approximation of price processes by simple stochastic
differential equations and semimartingales and the use of It\^{o}'s
calculus. However, modern markets are not represented by simple price
processes. Most orders happen on exchanges, where there are numerous bids,
offers, and less commonly, trades. Much activity in markets is concerned
with market making and the provision of liquidity; decisions to post to the
market are based closely on expectation of patterns of behaviour, and most
decisions are somewhat distant from any view about fundamental value. If one
is interested in alerting the trader who has a bug in his code, or
understanding how to trade a large order without excessive charges then the
semi-martingale model has a misplaced focus. 
\begin{figure}[H]
\centering
\includegraphics[
trim = 15mm 22mm 0mm 15mm, clip=true, width=1.07\textwidth]
{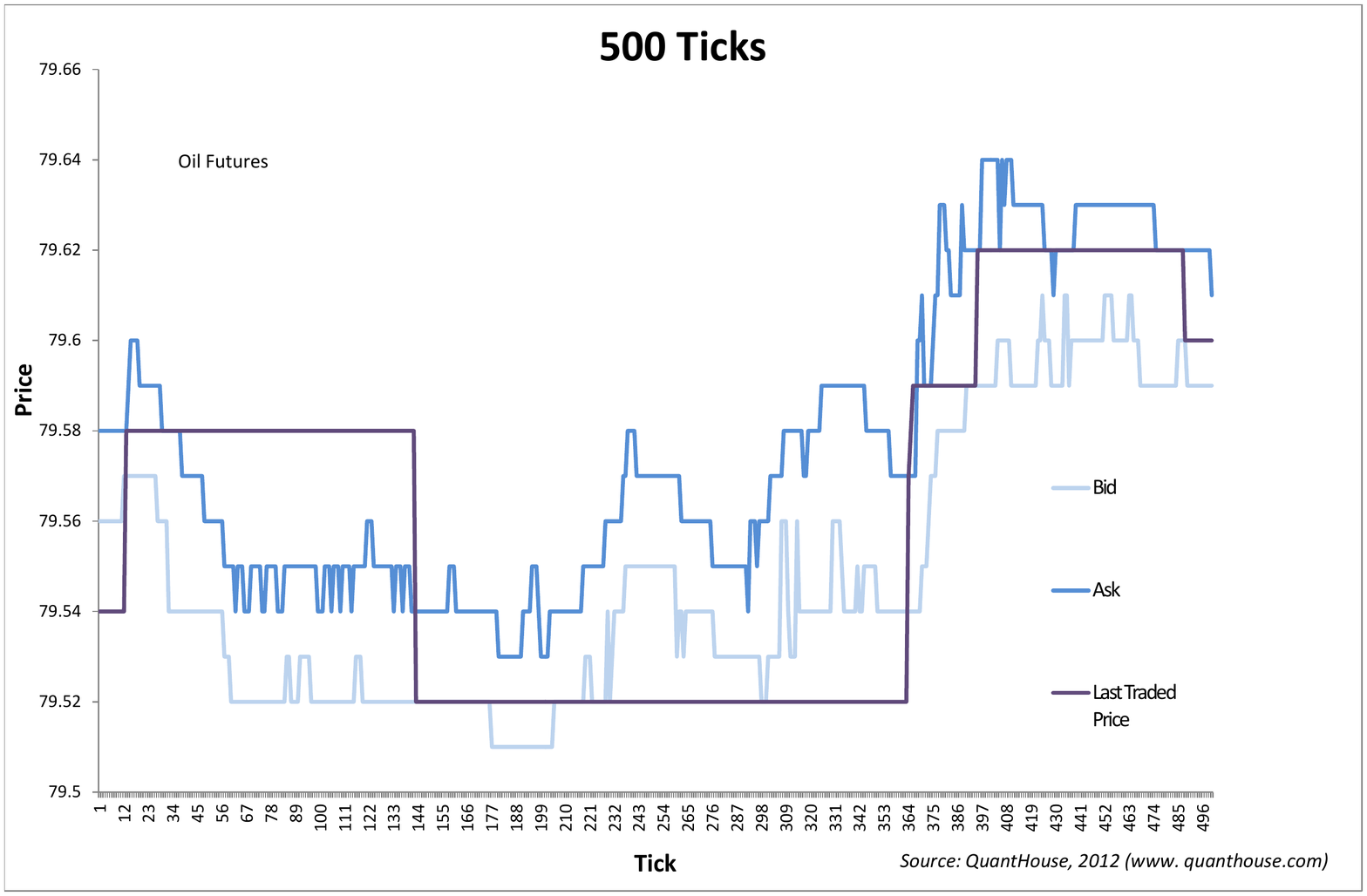}
\caption{A snapshot of level one order book data}
\label{fig:500Ticks}
\end{figure}
The data in the figure \ref{fig:500Ticks} is a snapshot of the level one order book showing activity on a market for oil futures over 500 changes (roughly a
15 minute period). One can see the bid and offer prices changing, although trades happen (and so the last executed price changes) much less frequently.
It is questionable whether a semi-martingale model for prices can capture this rich structure effectively.

\section{Paths - Simply Everywhere - Evolving systems}

Informally, a stream is a map $\gamma $ from a totally ordered set $I$ to
some state space, where we are interested in the effect (or transformation
of state) this stream achieves. As we have noted the same stream of
information can admit different representations with different fidelity.
When the totally ordered set $I$ is an interval and there are reasonable
path properties (e.g. such as right continuity) we will call the stream a
path. Nonetheless, many interesting streams are finite and discrete. There
are canonical and informative ways to convert them \cite%
{flint2013convergence} to continuous paths.

It is worth noting that, even at this abstract level, there are natural
mathematical operations and invariances that are applied to a stream. One
can reparameterise the speed at which one examines the stream and
simultaneously the speed at which one looks at the effects. One can split a
stream into two or more segments (a coproduct). One can sub-sample a stream.
In general we will focus on those streams which are presented in a way where
such sub-sampling degrades the information in the stream gradually. One can
also merge or interleave discrete streams according to their time stamps if
the totally ordered sets $I$, $I^{\prime }$can be interleaved. All of these
properties are inherited for the properties of totally ordered sets. If the
target "effect" or state space is linear there is also the opportunity to
translate and so concatenate streams or paths \cite{hambly2010uniqueness}
and so get richer algebraic structures. One of the most interesting and
economically important questions one can ask about a stream is how to
summarise (throw away irrelevant information) so as to succinctly capture
its effects. We give a few examples in Table \ref{tab:summarystream}.%\newline
\begin{table}[H]
\centering
\begin{tabular}{|c|c|c|}
\hline
text & schoolchild & precis \\ \hline
sound & audio engineer & faithful perception \\ \hline
web page & search provider & interest for reader \\ \hline
web click history & advertiser & effective ad placement \\ \hline
Brownian path & numerical analysis & effective simulation \\ \hline
rough paths & analyst & RDEs \\ \hline
\end{tabular}%
\caption{Examples of contexts where streams are summarised while retaining their essence.}
\label{tab:summarystream}
\end{table}
What is actually quite surprising is that there is a certain amount of
useful work one can do on this problem that does not depend on the nature of
the stream or path.

\section{A simple model for an interacting system}

We now focus on a very specific framework where the streams are maps from a
real interval, that we will intuitively refer to as the time domain, into an
a Banach space that we will refer to as the state space. We will work with
continuous paths in continuous time but, as we mentioned, there are
canonical ways to embed discrete tick style data into this framework using
the Hoff process and in financial contexts this is important. There is also
a more general theory dealing with paths with jumps [Williams, Simon].

\subsection{Controlled Differential Equations}

A path is a map $\gamma $ from an interval $J=\left[ J_{-},J_{+}\right] $
into a Banach space $E$. The dimension of $E$ may well be finite, but we
allow for the possibility that it is not. It has bounded ($p$-)variation if%
\begin{eqnarray*}
\sup_{\ldots u_{i}<u_{i+1}\ldots \in \left[ J_{-},J_{+}\right]
}\sum_{i}\left\Vert \gamma _{u_{i+1}}-\gamma _{u}\right\Vert &<&\infty \\
\sup_{\ldots u_{i}<u_{i+1}\ldots \in \left[ J_{-},J_{+}\right]
}\sum_{i}\left\Vert \gamma _{u_{i+1}}-\gamma _{u}\right\Vert ^{p} &<&\infty
\end{eqnarray*}%
where $p\geq 1$ In our context the path $\gamma $ is controlling the system,
and we are interested in its effect as measured by $y$ and the interactions
between $\gamma $ and $y$. It would be possible to use the theory of rough
paths to deal with the internal interactions of autonomous and "rough"
systems, one specific example of deterministic McKean Vlasov type is \cite%
{cass2013evolving}.

Separately there needs to be a space $F$ that carries the state of the
system and a family of different ways to evolve. We represent the dynamics
on $F$ through the space $\Omega \left( F\right) $ of vector fields on $F.$
Each vector field provides a different way for the state to evolve. We
connect this potential to evolve the state in $F$ to the control $\gamma $
via a linear map%
\[
V:E\overset{linear}{\rightarrow }\Omega \left( F\right) \text{.}
\]%
Immediately we can see the controlled differential equation%
\begin{eqnarray*}
dy_{t} &=&V\left( y_{t}\right) d\gamma _{t},\ y_{J_{-}}=a \\
\pi _{J}\left( y_{J_{-}}\right) &:&=y_{J_{+}}
\end{eqnarray*}%
provides a precise framework allowing for the system $y$ to respond to $%
\gamma $ according to the dynamics $V$. We call such a system a controlled
differential equation.

The model of a controlled differential equation is a good one. Many
different types of object can be positioned to fit the definition. Apart
from the more obvious applied examples, one can view a finite automata (in
computer science sense) and the geometric concept of lifting a path along a
connection as producing examples.

There are certain apparently trivial properties that controlled differential equations and
the paths that control them have; none the less they are structurally essential so we mention them now.

\begin{lemma}[Reparameterisation]
If $\tau :I\rightarrow J$ is an increasing homeomorphism, and if 
\[
dy_{t}=V\left( y_{t}\right) d\gamma _{t},\ y_{J_{-}}=a,
\]%
then the reparameterised control produces the reparameterised effect: 
\[
dy_{\tau \left( t\right) }=V\left( y_{\tau \left( t\right) }\right) d\gamma
_{\tau \left( t\right) },\ y_{\tau \left( I_{-}\right) }=a.
\]
\end{lemma}

\begin{lemma}[Splitting]
Let $\pi _{J}$ be the diffeomorphism capturing the transformational effect
of $\gamma |_{J}.$ Let $t\in J$. Then $\pi _{J}$ can be recovered by
composing the diffeomorpisms $\pi _{\left[ J_{-},t\right] }$, $\pi _{\left[
t,J_{+}\right] }$ associated with splitting the interval J at $t$ and
considering the composing the effect of $\gamma |_{_{\left[ J_{-},t\right]
}} $ and $\gamma |_{\left[ t,J_{+}\right] }$ separately:%
\[
\pi _{\left[ t,J_{+}\right] }\pi _{\left[ J_{-},t\right] }=\pi _{J}.
\]
\end{lemma}

In this way we see that, assuming the vector fields were smooth enough to
solve the differential equations uniquely and for all time, a controlled
differential equation is a homomorphism from the monoid of paths with
concatenation into the diffeomorphisms/transformations of the state space.
By letting $\pi $ act as an operator on functions we see that every choice
of $V$ defines a representation of the monoid of paths in $E$

\begin{remark}[Subsampling]
Although there is a good behaviour with respect to sub-sampling, which in
effect captures and quantifies the numerical analysis of these equations, it
is more subtle and we do not make it explicit here.
\end{remark}

\begin{remark}
Fixing $V$, restricting $\gamma $ to smooth paths on $\left[ 0,1\right] $
and considering the solutions $y$ with $y_{0}=a,$ generically the closure of
the set of pairs $\left( \gamma ,y\right) $ in the uniform topology is NOT
the graph of a map; $\gamma \rightarrow y$ is not closable and so is not
well defined as a (even an unbounded and discontinuous) function in the
space of continuous paths. Different approximations lead to different views
as to what the solution should be.
\end{remark}

\subsection{Linear Controlled Differential Equations}

Where the control $\gamma $ is fixed and smooth, the state space is linear,
and all the vector fields are linear, then the space of responses $y$, as
one varies the starting location $a$, is a linear space and $\pi _{\left[ S,T%
\right] }:a=y_{S}\rightarrow y_{T}$ is a linear automorphism. This case is
essentially Cartan's development of a path in a Lie Algebra into a path in
the Lie Group starting at the identity. From our point of view it is a very
important special case of our controlled differential equations; it reveals
one of the key objects we want to discuss in this paper.

Suppose $F$ is a Banach space, and $A$ is a linear map $E\rightarrow Hom_{%
\mathbb{R}}\left( F,F\right) $ and that $\gamma _{t}$ is a path in $E$.
Consider the linear differential equation%
\[
dy_{t}=Ay_{t}d\gamma _{t}.
\]%
By iterating using Picard iteration one obtains%
\[
y_{J_{+}}=\left( \sum_{n=0}^{\infty }A^{n}\underset{J_{-}\leq u_{1}\leq
\ldots \leq u_{n}\leq J_{+}}{\idotsint }d\gamma _{u_{1}}\otimes \ldots
\otimes d\gamma _{u_{n}}\right) y_{0}
\]

The Signature of $\gamma$ over the interval $J=\left[J_{-},J_{+}\right]$

\begin{definition}
The Signature $S$ of a bounded variation path (or more generally a weakly
geometric $p$-rough path) $\gamma $ over the interval $J=\left[ J_{-},J_{+}%
\right] $ is the tensor sequence%
\[
S\left( \gamma |_{J}\right) :=\sum_{n=0}^{\infty }\underset{u_{1}\leq \ldots
\leq u_{n}\in J^{n}}{\idotsint }d\gamma _{u_{1}}\otimes \ldots \otimes
d\gamma _{u_{n}}\in \bigoplus_{n=0}^{\infty }E^{\otimes n}
\]%
It is sometimes written $S\left( \gamma \right) _{J}$ or $S\left( \gamma
\right) _{J_{-},J_{+}}.$
\end{definition}

\begin{lemma}
The path $t\rightarrow S\left( \gamma \right) _{0,t}$ solves a linear
differential equation controlled by $\gamma $.
\end{lemma}

\begin{proof}
The equation is the universal non-commutative exponential: 
\begin{eqnarray*}
dS_{0,t} &=&S_{0,t}\otimes d\gamma _{t}. \\
S_{0,0} &=&1
\end{eqnarray*}
\end{proof}

The solution to any linear equation is easily expressed in terms of the
Signature%
\begin{eqnarray}
dy_{t} &=&Ay_{t}d\gamma _{t}  \nonumber \\
y_{_{J_{+}}} &=&\left( \sum_{0}^{\infty }A^{n}S_{J}^{n}\right) y_{_{J_{-}}}
\label{chensig2} \\
\pi _{J} &=&\sum_{0}^{\infty }A^{n}S_{J}^{n}  \nonumber
\end{eqnarray}%
and we will see in the next sections that this series converges very well
and even the first few terms in $S$ are effective in describing the response 
$y_{T}$ leading to the view that $\gamma |_{J}\rightarrow S\left( \gamma
|_{J}\right) $ is a transform with some value. The use of $S$ to describe
solutions to linear controlled differential equations goes back at least to
Chen, and Feynman. The \emph{magic} is that one can estimate the errors in
convergence of the series (\ref{chensig2}) without detailed understanding of 
$\gamma $ or $A.$

\section{Remarkable Estimates (for $p>1$)}

It seems strange, and even counter intuitive, that one should be able to
identify and abstract a finite sequence of features or coefficients
describing $\gamma $ adequately so that its effect on a broad range of
different systems could be accurately predicted without detailed knowledge
of the system $A$ or the path $\gamma $ - beyond those few coefficients. But
that is the truth of it, there are easy uniform estimates capturing the
convergence of the series (\ref{chensig2}) based entirely on the length (or
more generally $p$-rough path variation) of the control and the norm of $A$
as a map from $E$ to the linear vector fields on $F$.

\begin{lemma}
If $\gamma $ is a path of finite variation on $J\ $with length $\left\vert
\gamma _{J}\right\vert <\infty $, then%
\begin{eqnarray*}
S_{J}^{n} &:&=\underset{u_{1}\leq \ldots \leq u_{n}\in J^{n}}{\idotsint }%
d\gamma _{u_{1}}\otimes \ldots \otimes d\gamma _{u_{n}} \\
&\leq &\frac{\left\vert \gamma _{J}\right\vert ^{n}}{n!}
\end{eqnarray*}%
giving uniform error control%
\[
\left\Vert y_{J_{+}}-\sum_{0}^{N-1}A^{n}\underset{J_{-}\leq u_{1}\leq \ldots
\leq u_{n}\leq J_{+}}{\idotsint }d\gamma _{u_{1}}\otimes \ldots \otimes
d\gamma _{u_{n}}y_{0}\right\Vert \leq \left( \sum_{n=N}^{\infty }\frac{%
\left\Vert A\right\Vert ^{n}\left\vert \gamma _{J}\right\vert ^{n}}{n!}%
\right) \left\Vert y_{0}\right\Vert .
\]
\end{lemma}

\begin{proof}
Because the Signature of the path always solves the characteristic
differential equation it follows that one can reparameterise the path $%
\gamma $ without changing the Signature of $\gamma $. Reparameterise $\gamma 
$ so that it is defined on an interval $J$ of length $\left\vert \gamma
\right\vert $ and runs at unit speed. Now there are $n!$ disjoint simplexes
inside a cube obtained by different permuted rankings of the coordinates and
thus 
\begin{eqnarray*}
\left\Vert S_{J}^{n}\right\Vert &:&=\left\Vert \underset{u_{1}\leq \ldots
\leq u_{n}\in J^{n}}{\idotsint }d\gamma _{u_{1}}\otimes \ldots \otimes
d\gamma _{u_{n}}\right\Vert \\
&=&\left\Vert \underset{u_{1}\leq \ldots \leq u_{n}\in J^{n}}{\idotsint }%
\dot{\gamma}_{u_{1}}\otimes \ldots \otimes \dot{\gamma}_{u_{n}}du_{1}\ldots
du_{n}\right\Vert \\
&=&\underset{u_{1}\leq \ldots \leq u_{n}\in J^{n}}{\idotsint }\left\Vert 
\dot{\gamma}_{u_{1}}\otimes \ldots \otimes \dot{\gamma}_{u_{n}}\right\Vert
du_{1}\ldots du_{n} \\
&=&\underset{u_{1}\leq \ldots \leq u_{n}\in J^{n}}{\idotsint }du_{1}\ldots
du_{n} \\
&=&\frac{\left\vert \gamma _{J}\right\vert ^{n}}{n!}.
\end{eqnarray*}%
from which the second estimate is clear.
\end{proof}

The Poisson approximation of a normal distribution one learns at high school
ensures that the estimates on the right become very sharply estimated in
terms of $\lambda \rightarrow \infty $ and pretty effective as soon as $%
N\geq \left\Vert A\right\Vert \left\vert \gamma _{J}\right\vert +\lambda 
\sqrt{\left\Vert A\right\Vert \left\vert \gamma _{J}\right\vert }.$

\begin{remark}
The uniform convergence of the series 
\[
\sum_{n=0}^{N-1}A^{n}\underset{J_{-}\leq u_{1}\leq \ldots \leq u_{n}\leq
J_{+}}{\idotsint }d\gamma _{u_{1}}\otimes \ldots \otimes d\gamma
_{u_{n}}y_{0}
\]
and the obvious continuity of the terms of the series in the inputs $\left(
A,\gamma ,y_{0}\right) $ guarantees that the response $y_{T}$ is jointly
continuous (uniform limits of continuous functions are continuous) in $%
\left( A,\gamma ,y_{0}\right) $ where $\gamma $ is given the topology of $1$%
-variation (or any of the rough path metrics). It is already the case that 
\[
\gamma \rightarrow \underset{J_{-}\leq u_{1}\leq u_{2}\leq J_{+}}{\idotsint }%
d\gamma _{u_{1}}\otimes d\gamma _{u_{2}}
\]%
fails the closed graph property in the uniform metric.
\end{remark}

\section{The Log Signature}

It is easy to see that the Signature of a path segment actually takes its
values in a very special curved subspace of the tensor algebra. Indeed, Chen
noted that the map $S$ is a homomorphism of path segments with concatenation
into the algebra, and reversing the path segment produces the inverse
tensor. As a result one sees that the range of the map is closed under
multiplication and has inverses so it is a group (inside the grouplike
elements) in the tensor series. It is helpful to think of the range of this
Signature map as a curved space in the tensor series. As a result there is a
lot of valuable structure. One important map is the logarithm; it is one to
one on the group and provides a flat parameterisation of the group in terms
of elements of the free Lie series.

\begin{definition}
If $\gamma _{t}\in E$ is a path segment and $S$ is its Signature then 
\begin{eqnarray*}
S &=&1+S^{1}+S^{2}+\ldots \ \forall i,\ S^{i}\in E^{\otimes i} \\
\log \left( 1+x\right) &=&x-x^{2}/2+\ldots \\
\log S &=&\left( S^{1}+S^{2}+\ldots \right) -\left( S^{1}+S^{2}+\ldots
\right) ^{2}/2+\ldots
\end{eqnarray*}%
The series $\log S=\left( S^{1}+S^{2}+\ldots \right) -\left(
S^{1}+S^{2}+\ldots \right) ^{2}/2+\ldots $ which is well defined, is
referred to as the log Signature of $\gamma .$
\end{definition}

Because the space of tensor series $T\left( \left( E\right) \right)
:=\bigoplus_{0}^{\infty }E^{\otimes n}$ is a unital associative algebra
under $\otimes ,+$ it is also a Lie algebra, and with $\left[ A,B\right]
:=A\otimes B-B\otimes A.$

\begin{definition}
There are several canonical Lie algebras associated to $T\left( \left(
E\right) \right) $; we use the notation $\mathcal{L}\left( E\right) $ for
the algebra generated by $E$ (the space of Lie polynomials), $\mathcal{L}%
^{\left( n\right) }\left( E\right) $ the projection of this into $T^{\left(
n\right) }\left( E\right) =T\left( \left( E\right) \right)
/\bigoplus_{n+1}^{\infty }E^{\otimes m}$ (the Lie algebra of the free
nilpotent group $G^{n}$ of n steps) and $\mathcal{L}\left( \left( E\right)
\right) $ the projective limit of the $\mathcal{L}^{\left( n\right) }\left(
E\right) $ (the Lie Series).
\end{definition}

Because we are working in characteristic zero, we may take the exponential,
and this recovers the Signature, so no information is lost. A key
observation of Chen \cite{chen1957integration}was that if $\gamma $ is a
path segment then $\log S\left( \gamma \right) \in \mathcal{L}\left( \left(
E\right) \right) $. The map from paths \cite%
{Rashevski1938aboutconecting,chow1939sisteme}to $\mathcal{L}^{\left(
n\right) }\left( E\right) $ via the projection $\pi _{n}:T\left( \left(
E\right) \right) \rightarrow T^{\left( n\right) }\left( E\right) $ is onto.
Up to equivalence under a generalised notion of reparameterisation of paths
known as treelike equivalence, the map from paths $\gamma $ of finite length
in $E$ to their Signatures $S\left( \gamma \right) \in T\left( \left(
E\right) \right) $ or log-Signatures $\log S\in \mathcal{L}\left( \left(
E\right) \right) $ is injective \cite{hambly2010uniqueness}. Treelike
equivalence is an equivalence relation on paths of finite variation, each
class has a unique shortest element, and these tree reduced paths form a
group. However the range of the log-Signature map in $\mathcal{L}\left(
\left( E\right) \right)$, although well behaved under integer multiplication
is not closed under integer division \cite{lyonssidorova2006ontheradius} and
so the Lie algebra of the group of tree reduced paths is well defined but
not a linear space; it is altogether a more subtle object.

Implicit in the definition of a controlled differential equation%
\[
dy_{t}=f\left( y_{t}\right) d\gamma _{t},\ y_{0}=a
\]%
is the map $f$. This object takes an element $e\in E$ and an element $y\in F$
and produces a second vector in $F$, representing the infinitesimal change
to the state $y$ of the system that will occur if $\gamma $ is changed
infinitesimally in the direction $e$. This author is clear that the best way
to think about $f$ is as a linear map from the space $E$ into the vector
fields on $F$. In this way one can see that the integral of $f$ along $%
\gamma $ in its simplest form is a path in the Lie algebra and that in
solving the differential equation we are developing that path into the
group. Now, at least formally, the vector fields are a Lie algebra (for the
diffemorphisms of F) and subject to the smoothness assumptions we can take
Lie brackets to get new vector fields. Because $\mathcal{L}\left( \left(
E\right) \right) $ is the free Lie algebra over $E$ (Chapter II, \cite%
{bourbaki1975lie}) any linear map $f$ of $E$ into a Lie algebra $\mathfrak{g}
$ induces a unique Lie map extension $f_{\ast }$ to a Lie map from $\mathcal{%
L}\left( \left( E\right) \right) $ to $\mathfrak{g}$. This map can be
readily implemented and is well defined because of the abstract theory%
\begin{eqnarray*}
e &\rightarrow &f\left( e\right) ~~\ \text{a vector field} \\
e_{1}e_{2}-e_{2}e_{1} &\rightarrow &f\left( e_{1}\right) f\left(
e_{2}\right) -f\left( e_{2}\right) f\left( e_{1}\right) ~~\ \text{a vector
field} \\
\tilde{f} &:&\mathcal{L}^{\left( n\right) }\left( E\right) \rightarrow \text{%
vector fields.}
\end{eqnarray*}%
although in practice one does not take the map to the full projective limit.

\section{The ODE method}

The linkage between truncations of the log-Signature in $\mathcal{L}\left(
\left( E\right) \right) $ and vector fields on $Y$ is a practical one for
modelling and understanding controlled differential equations. It goes well
beyond theory and underpins some of the most effective and stable numerical
approaches (and control mechanisms) for translating the information in the
control $\gamma $ into information about the response.

If $dy_{t}=f\left( y_{t}\right) d\gamma _{t}$, and $y_{J_{-}}=a$ then how
can we use the first few terms of the (log-)Signature of $\gamma $ to
provide a good approximation to $y_{J_{+}}$? We could use picard iteration,
or better an euler method based on a Taylor series in terms of the
Signatures. Picard iteration for $\exp z$ already illustrates one issue.
Picard interation yields a power series as approximation - fine if $z=100, $%
but awful if $x=-100$. However, there is a more subtle problem to do with
stability that almost all methods based on Taylor series have - stability -
they can easily produce approximations that are not feasible. These are
aggravated in the controlled case because of the time varying nature of the
systems. It can easily happen that the solutions to the vector fields are
hamiltonian etc. The ODE method uses the first few terms of the Signature to
construct a time invariant ODE (vector field) that if one solves it for unit
time, it provides an approximation to the desired solution. It pushes the
numerics back onto state of the art ODE solvers. Providing the ODE solver is
accurate and stable then the approximation to $y$ will also be. One can use
symplectic solvers etc. At the level of rough paths, the approximation is
obtained by replacing the path $\gamma $ with a new rough path $\hat{\gamma}$
(a geodesic in the nilpotent group $G^{n})$ with the same first few terms in
the Signature; this guarantees the feasibility of the approximations. Today,
rough path theory can be used to estimate the difference between the
solution and the approximation in terms of the distance between $\gamma $
and $\hat{\gamma}$ even in infinite dimensions.\cite{castell1995efficient}%
\cite{boutaib2013dimension}

\begin{remark}
A practical numerical scheme can be built as follows.

\begin{enumerate}
\item Describe $\gamma $ over a short interval $J$ in terms of first few
terms of $\log S\left( \gamma _{\left[ J_{-},J_{+}\right] }\right) $
expressed as a linear combination of terms of a fixed hall basis:%
\begin{eqnarray*}
\log S_{J} &=&l^{1}+l^{2}+\ldots \in \mathcal{L}\left( \left( E\right)
\right) \\
l^{\left( n\right) } &=&\pi _{n}\left( \log S_{J}\right) =l^{1}+\ldots
+l^{n}\in \mathcal{L}^{\left( n\right) }\left( E\right) \\
l^{1} &=&\sum_{i}\lambda _{i}e_{i} \\
l^{2} &=&\sum_{i<j}\lambda _{ij}\left[ e_{i},e_{j}\right] , \\
&&\text{\ldots }
\end{eqnarray*}%
and use this information to produce a path dependent vector field $V=\tilde{f%
}\left( l^{\left( n\right) }\right) .$

\item Use an appropriate ODE solver to solve the ODE $\dot{x}_{t}=V\left(
x_{t}\right) $, where $x_{0}=y_{J_{-}}$. A stable high order approximation
to $y_{J_{+}}$ is given by $x_{J_{+}}$.

\item Repeat over small enough time steps for the high order approximations
to be effective.

\item The method is high order, stable, and corresponding to replacing $%
\gamma $ with a piecewise geodesic path on successively finer scales.
\end{enumerate}
\end{remark}

\section{Going to Rough Paths}

As this is a survey, we have deliberately let the words rough path enter the
text before they are introduced more formally. Rough path theory answers the
following question. Suppose that $\gamma $ is a smooth path but still on
normal scales, a highly rough and oscillatory path. Suppose that we have
some smooth system $f$ . Give a simple metric on paths $\gamma $ and a
continuity estimate that ensures that if two paths that are close in this
metric then their responses are quantifiably close as well. The estimate
should only depend on $f$ through its smoothness. There is such a theory 
\cite{lyons2007differential}, and a family of rough path metrics which make
the function $\gamma \rightarrow y$ uniformly continuous. The completion of
the smooth paths $\gamma $ under these metrics are the rough paths we speak
about. The theory extends to an infinite dimensional one and the estimates
are uniform in a way that does not depend on dimension.

There are many sources for this information on rough paths for different
kinds of audience and we do not repeat that material. We have mentioned that
two smooth paths have quantifiable close responses to a smooth $f$ over a
fixed time interval if the first terms in the Signature agree over this time
interval. We can build this into a metric:%
\[
d_{p}\left( \gamma |_{J},\hat{\gamma}|_{J}\right) =\sup_{J_{-}\leq u_{1}\leq
\ldots \leq u_{n}\leq J_{+}}\sum_{i}\max_{m\leq \left\lfloor p\right\rfloor
}\left\Vert S^{m}\left( \gamma |_{\left[ u_{i},u_{i+1}\right] }\right)
-S^{m}\left( \hat{\gamma}|_{\left[ u_{i},u_{i+1}\right] }\right) \right\Vert
^{p/m}
\]%
and providing the system is $Lip\left( p+\varepsilon \right) $ the response
will behave uniformly with the control. The completion of the piecewise
smooth paths under $d_{p}$ are $p$-variation paths. They do not have
smoothness but they do have a "top down" description and can be viewed as
living in a $\left\lfloor p\right\rfloor $-step nilpotent group over $E.$

It is worth distinguishing the Kolmogorov and the rough path view on paths.
In the former, one considers fixed times $t_{i}$, open sets $O_{i}$, and
considers the probability that for all $i$, $x_{t_{i}}\in O_{i}$. In other
words the emphasis is on where the path is at given times. This gated
description will never capture the rough path; parameterisation is
irrelevant but increments over small intervals $\left[ u_{i},u_{i+1}\right] $%
, are critical. More accurately one describes a path through an examination
of the effect of it's path segment into a simple nonlinear system (the lift
onto a nilpotent group). Knowing this information in an analytically
adequate way is all one needs to know to predict the effect of the path on a
general system.

The whole rough path theory is very substantial and we cannot survey it
adequately here. The range is wide, and is related to any situation where
one has a family of non-commuting operators and one wants to do analysis on
apparently divergent products and for example it is interesting to
understand the paths one gets as partial integrals of complex Fourier
transform as the nonlinear Fourier transform is a differential equation
driven by this path. Some results have been obtained in this direction \cite%
{lyons2013partial} while the generalisations to spatial contexts are so huge
that they are spoken about elsewhere at this congress. Many books are now
written on the subject \cite{friz2010multidimensional}.and new lecture notes
by Friz are to appear soon with recent developments. So in what is left of
this paper we will focus on one topic the Signature of a path and the
expected Signature of the path with a view to partially explaining how it is
really an extension of Taylor's theorem to various infinite dimensional
groups, and how we can get practical traction from this perspective. One key
point we will not mention is that using Taylor's theorem twice works! This
is actually a key point that the whole rough path story depends on and which
validates its use. One needs to read the proofs to understand this
adequately and, except for this sentence, suppress it completely here.

\section{Coordinate \textbf{Iterated} Integrals}

In this short paper we have to have a focus, and as a result we cannot
explore the analysis and algebra needed to fully describe rough paths or to
discuss the spatial generalisations directly even though they are having
great impact\cite{hairer2014regularity}\cite{hairer2014theory}. Nonetheless
much of what we say can be though of as useful foundations for this work. We
are going to focus on the Signature as a tool for understanding paths and as
a new tool to help with machine learning.

The essential remark may seem a bit daunting to an analyst, but will be
standard to others. \emph{The dual of the enveloping algebra of a
group(like) object has a natural abelian product structure and linearises
polynomial functions on a group.} This fact allows one to use linear
techniques on the linear spaces to approximate generic smooth (and
nonlinear) functions on the group. Here the group is the "group" of paths.

Monomials are special functions on $\mathbb{R}^{n}$, and polynomials are
linear combinations of these monomials. Because monomials span an algebra,
the polynomials are able to approximate any continuous function on a compact
set. Coordinate iterated integrals are linear functionals on the tensor
algebra and at the same time they are the monomials or the features on path
space.

\begin{definition}
Let $\boldsymbol{e}=e_{1}\otimes \ldots \otimes e_{n}\in \left( E^{\ast
}\right) ^{\otimes n}\subset T\left( E^{\ast }\right) $, and $\phi _{%
\boldsymbol{e}}\left( \gamma \right) :=\left\langle \boldsymbol{e,}S\left(
\gamma \right) \right\rangle $ then we call $\phi _{\boldsymbol{e}}\left(
\gamma \right) $ a coordinate iterated integral.
\end{definition}

\begin{remark}
Note that $S\left( \gamma \right) \in T\left( \left( E\right) \right)
=\bigoplus_{0}^{\infty }E^{\otimes n}$ and%
\begin{eqnarray*}
\phi _{\boldsymbol{e}}\left( \gamma \right) &=&\left\langle \boldsymbol{e,}%
S\left( \gamma \right) \right\rangle \\
&=&\underset{u_{1}\leq \ldots \leq u_{n}\in J^{n}}{\idotsint }\left\langle
e_{1},d\gamma _{u_{1}}\right\rangle \ldots \left\langle e_{n},d\gamma
_{u_{n}}\right\rangle
\end{eqnarray*}%
justifying the name. $\phi _{\boldsymbol{e}}$ is a real valued function on
Signatures of paths.
\end{remark}

\begin{lemma}
The shuffle product $\amalg $ on $T\left( E^{\ast }\right) $ makes $T\left(
E^{\ast }\right) $ a commutative algebra and corresponds to point-wise
product of coordinate integrals 
\[
\phi _{\boldsymbol{e}}\left( \gamma \right) \phi _{\boldsymbol{f}}\left(
\gamma \right) =\phi _{\boldsymbol{e\amalg f}}\left( \gamma \right)
\]
\end{lemma}

This last identity, which goes back to Ree, is important because it says
that if we consider two linear functions on $T\left( \left( E\right) \right) 
$ and multiply them together then their product - which is quadratic
actually agrees with a linear functional on the group like elements. The
shuffle product identifies the linear functional that does the job.

\begin{lemma}
Coordinate iterated integrals, as features of paths, span an algebra that
separates Signatures and contains the constants.
\end{lemma}

This lemma is as important for understanding smooth functions on path spaces
as monomials are for understanding smooth functions on $\mathbb{R}^{n}.$%
There are only finitely many of each degree if $E$ is finite dimensional
(although the dimension of the spaces grow exponentially) \cite%
{lyons2007differential}. We will see later that this property is important
for machine learning and nonlinear regression applications but first we want
to explain how the same remark allows one to understand measures on paths
and formulate the notion of Fourier and Laplace transform.

\section{Expected Signature}

The study of the expected Signature was initiated by Fawcett in his thesis 
\cite{fawcett2002problems}. He proved

\begin{proposition}
Let $\mu $ be a compactly supported probability measure on paths $\gamma $
with Signatures in a compact set $K$. Then $\hat{S}=\mathbb{E}_{\mu }\left(
S\left( \gamma \right) \right) $ uniquely determines the law of $S\left(
\gamma \right) .$
\end{proposition}

\begin{proof}
Consider $\mathbb{E}_{\mu }(\phi _{\boldsymbol{e}}\left( \gamma \right) ).$ 
\begin{eqnarray*}
\mathbb{E}_{\mu }(\phi _{\boldsymbol{e}}\left( \gamma \right) ) &=&\mathbb{E}%
_{\mu }\left( \left\langle \boldsymbol{e,}S\left( \gamma \right)
\right\rangle \right) \\
&=&\left\langle \boldsymbol{e,}\mathbb{E}_{\mu }\left( S\left( \gamma
\right) \right) \right\rangle \\
&=&\left\langle \boldsymbol{e,}\hat{S}\right\rangle
\end{eqnarray*}%
Since the $\boldsymbol{e}$ with the shuffle product form an algebra and
separate points of $K$ the Stone-Weierstrass Theorem implies they form a
dense subspace in $C\left( K\right) $ and so determine the law of the
Signature of $\gamma $.
\end{proof}

Given this lemma it immediately becomes interesting to ask how does one
compute $\mathbb{E}_{\mu }\left( S\right) $. Also, $\mathbb{E}_{\mu }\left(
S\right) $ is like a Laplace transform and will fail to exist for reasons of
tail behaviour of the random variables. Is there a characteristic function?
Can we identify the general case where the expected Signature determines the
law in the non-compact case. All of these are fascinating and important
questions. Partial answers and strong applications are emerging. One of the
earliest was the realisation that one could approximate effectively to a
complex measure such as Wiener measure by a measure on finitely many paths
that has the same expected Signature on $T^{\left( n\right) }\left( E\right) 
$\cite{lyons2004cubature,litterer2011cubature}.

\section{Computing expected Signatures}

Computing Laplace and Fourier transforms can often be a challenging problem
for undergraduates. In this case suppose that $X$ a Brownian motion with L%
\'{e}vy area on a bounded $C^{1}$ domain $\Omega \subset \mathbb{R}^{d},$%
stopped on first exit. The following result explains how one may construct
the expected Signature as a recurrence relation in PDEs\cite%
{ExpectedSignatureBM}.

\begin{theorem}
Let%
\begin{eqnarray*}
F\left( z\right) &:&=\mathbb{E}_{z}\left( S\left( X|_{\left[ 0,T_{\Omega }%
\right] }\right) \right) \\
F &\in &S\left( \left( \mathbb{R}^{d}\right) \right) \\
F &=&\left( f_{0},f_{1},\ldots ,\right)
\end{eqnarray*}%
Then $F$ satisfies and is determined by a PDE finite difference operator%
\begin{eqnarray*}
\Delta f_{n+2} &=&-\sum_{i=1}^{d}e_{i}\otimes e_{i}\otimes
f_{n}-2\sum_{i=1}^{d}e_{i}\otimes \frac{\partial }{\partial z_{i}}f_{n+1} \\
f_{0} &\equiv &1,\ f_{1}\equiv 0,\text{ and}\ f_{j}|_{\partial \Omega
}\equiv 0,\ j>0
\end{eqnarray*}
\end{theorem}

Combining this result with Sobolev and regularity estimates from PDE theory
allow one to extract much nontrivial information about the underlying
measure although it is still open whether in this case the expected
Signature determines the measure. This question is difficult even for
Brownian motion on $\min(T_{\tau}, t)$ although (unpublished) it looks as if
the question can be resolved.

Other interesting questions about expected Signatures can be found for
example in \cite{boedihardjo2013uniqueness}.

\section{Characteristic Functions of Signatures}

It is possible to build a characteristic function out of the expected
Signature by looking at the linear differential equations corresponding to
development of the paths into finite dimensional unitary groups. These
linear images of the Signature are always bounded and so expectations always
make sense.

Consider $SU\left( d\right) \subset M\left( d\right) $ and realise $su\left(
d\right) $ as the space of traceless Hermitian matrices and consider 
\begin{eqnarray*}
\psi &:&E\rightarrow su\left( d\right) \\
d\Psi _{t} &=&\psi \left( \Psi _{t}\right) d\gamma _{t}.
\end{eqnarray*}

Essential features of the co-ordinate interated integrals included that they were linear functions on the 
tensor algebra, that they were real valued functions that separated signatures, and that they spanned an algebra.

It is core to rough path theory that any representation of paths via a linear controlled equation can also be regarded as a linear function and that products can also be represented as sums. If one can show that products associated to the finite dimensional unitary groups can be expressed as sums of finite linear combinations of finite dimensional unitary representations, and add an appropriate topology on grouplike elements, one can repeat the ideas outlined above but now with expectations that always exist and obtain the analogue of characteristic function.

\begin{theorem}
$\Psi _{t}$ is a linear functional on the tensor algebra restricted to the
Signatures $S\left( \gamma |_{\left[ 0,t\right] }\right) $ and is given by a
convergent series. It is bounded and so its expectation as $\gamma $ varies
randomly always makes sense. The function $\psi \rightarrow \mathbb{E}\left(
\Psi _{J_{+}}\left( S\right) \right) $ is an extended characteristic
function.
\end{theorem}

\begin{proposition}
$\psi \rightarrow \Psi \left( S\right) $ (polynomial identities of Gambruni
and Valentini) span an algebra and separate Signatures as $\psi $ and $d$
vary.
\end{proposition}

\begin{corollary}
The laws of measures on Signatures are completely determined by $\psi
\rightarrow \mathbb{E}\left( \Psi \left( S\right) \right) $
\end{corollary}

\begin{proof}
Introduce a polish topology on the grouplike elements.
\end{proof}

These results can be found in \cite{chevyrev2014unitary}, the paper also
gives a sufficient condition for the expected Signature to determine the law
of the underlying measure on Signatures.

\section{Moments are complicated}

The question of determining the Signature from its moments seems quite hard
at the moment.

\begin{example}
Observe that if $X$ is $N\left( 0,1\right) $ then although $X^{3}$ is not
determined by its moments, if $Y=X^{3}$ then $\left( X,Y\right) $ is. The
moment information implies $\mathbb{E}\left( \left( Y-X^{3}\right)
^{2}\right) =0.$
\end{example}

We repeat our previous question. Does the expected Signature determine the
law of the Signature for say stopped Brownian motion. The problem seems to
capture the challenge.

\begin{lemma}[\protect\cite{chevyrev2014unitary}]
If the radius of convergence of $\sum z^{n}\mathbb{E}\left\Vert
S^{n}\right\Vert $ is infinite then the expected Signature determines the
law.
\end{lemma}

\begin{lemma}[\protect\cite{ExpectedSignatureBM}]
If $X$ a Brownian motion with L\'{e}vy area on a bounded $C^{1}$ domain $%
\Omega \subset \mathbb{R}^{d}$ then $\sum z^{n}\mathbb{E}\left\Vert
S^{n}\right\Vert $ has at the least a strictly positive lower bound on the
radius of curvature.
\end{lemma}

The gap in understanding between the previous two results is, for the
author, a fascinating and surprising one that should be closed!

\section{Regression onto a feature set}

Learning how to regress or learn a function from examples is a basic problem
in many different contexts. In what remains of this paper, we will outline
recent work that explains how the Signature engages very naturally with this
problem and why it is this engagement that makes it valuable in rough path
theory too.

We should emphasise that the discussion and examples we give here is at a
very primitive level of fitting curves. We are not trying to do statistics,
or model and make inference about uncertainty. Rather we are trying to solve
the most basic problems about extracting relationships from data that would
exist even if one had perfect knowledge. We will demonstrate that this
approach can be easy to implement and effective in reducing dimension and
doing effective regression. We would expect Baysian statistics to be an
added layer added to the process where uncertanty exists in the data that
can be modelled reasonably.

A core idea in many successful attempts to learn functions from a collection
of known (point, value) pairs revolves around the identification of basic
functions or features that are readily evaluated at each point and then try
to express the observed function a\emph{\ linear} combination of these basic
functions. For example one might evaluate a smooth function $\rho $ at a
generic collection $\left\{ x_{i}\in \left[ 0,1\right] \right\} $ of points
producing pairs $\left\{ \left( y_{i}=\rho \left( x_{i}\right) ,x_{i}\right)
\right\} $ Now consider as feature functions $\left\{ \phi _{n}:x\rightarrow
x^{n},n=0,\ldots N\right\} $. These are certainly easy to compute for each $%
x_{i}$. We try to express 
\[
\rho \simeq \sum_{n=0}^{N}\lambda _{n}\phi _{n}
\]%
and we see that if we can do this (that is to say $\rho $ is well
approximated by a polynomial) then the $\lambda _{n}$ are given by the
linear equation%
\[
y_{j}=\sum_{n=0}^{N}\lambda _{n}\phi _{n}\left( x_{j}\right) .
\]%
In general one should expect, and it is even desirable, that the equations
are significantly degenerate. The purpose of learning is presumably to be
able to use the function $\sum_{n=0}^{N}\lambda _{n}\phi _{n}$ to predict $%
\rho $ on new and unseen values of $x$ and to at least be able to replicate
the observed values of $y$.

There are powerful numerical techniques for identifying robust solutions to
these equations. Most are based around least squares and singular value
decomposition, along with ~$L^{1}$ constraints and Lasso.

However, this approach fundamentally depends on the assumption that the $%
\phi _{n}$ span the class of functions that are interesting. It works well
for monomials because they span an algebra and so every $C^{n}\left(
K\right) $ function can be approximated in $C^{n}\left( K\right) $ by a
multivariate real polynomial. It relies on a priori knowledge of smoothness
or Lasso style techniques to address over-fitting.

I hope the reader can now see the significance of the coordinate iterated
integrals. If we are interested in functions (such as controlled
differential equations) that are effects of paths or streams, then we know
from the general theory of rough paths that the functions are indeed well
approximated locally by linear combinations of coordinate iterated integrals
. Coordinate iterated integrals are a natural feature set for capturing the
aspects of the data that predicting the effects of the path on a controlled
system.

The shuffle product ensures that linear combinations of coordinate iterated
integrals are an algebra which ensures they span adequately rich classes of
functions. We can use the classical techniques of non-linear interpolation
with these new feature functions to learn and model the behaviour of systems.

In many ways the machine learning perspective explains the whole theory of
rough paths. If I\ want to model the effect of a path segment, I can do a
good job by studying a few set features of my path locally. On smaller
scales the approximations improve since the functionals the path interacts
with become smoother. If the approximation error is small compared with the
volume, and consistent on different scales, then knowing these features, and
only these features, on all scales describes the path or function adequately
enough to allow a limit and integration of the path or function against a
Lipchitz function.

\section{The obvious feature set for streams}

The feature set that is the coordinate iterated integrals is able (with
uniform error - even in infinite dimension) via linear combinations whose
coefficients are derivatives of $f$, to approximate solutions to controlled
differential equations \cite{boutaib2013dimension}. In other words, any
stream of finite length is characterised up to reparameterisation by its log
Signature (see \cite{hambly2010uniqueness}) and the Poincare-Birkhoff-Witt
theorem confirms that the coordinate iterated integrals are one way to
parameterise the polynomials on this space. Many important nonlinear
functions on paths are well approximated by these polynomials...

We have a well defined methodology for linearisation of smooth functions on
unparameterised streams as linear functionals of the Signature. As we will
explain in the remaining sections, this has potential for practical
application even if it comes from the local embedding of a group into its
enveloping algebra and identifying the dual with the real polynomials and
analytic functions on the group.

\section{Machine learning, an amateur's first attempt}

Applications do not usually have a simple fix but require several methods in
parallel to achieve significance. The best results to date for the use of
Signatures have involved the recognition of Chinese characters \cite%
{yin2013icdar} where Ben Graham put together a set of features based loosely
on Signatures and state of the art deep learning techniques to win a
worldwide competition organised by the Chinese Academy of Sciences.

We will adopt a different perspective and simply explain a very transparent
and naive approach, based on Signatures, can achieve with real data. The
work appeared in \cite{gyurko2013extracting}. The project and the data
depended on collaboration with commercial partners acknowledged in the paper
and is borrowed from the paper.

\subsection{classification of time-buckets from standardised data}

We considered a simple classification learning problem. We considered a
moderate data set of 30 minutes intervals of normalised one minute financial
market data, which we will call buckets. The buckets are distinguished by
the time of day that the trading is recorded. The buckets are divided into
two sets - a learning and a backtesting set. The challenge is simple: learn
to distinguish the time of day by looking at the normalised data (if indeed
one can - the normalisation is intended to remove the obvious). It is a
simple classification problem that can be regarded as learning a function
with only two values 
\[
\begin{array}{ccc}
f\left( \text{time series}\right) & \rightarrow & \text{time slot} \\ 
f\left( \text{time series}\right) =1 &  & \text{time slot=10.30-11.00} \\ 
f\left( \text{time series}\right) =0 &  & \text{time slot=14.00-14.30}%
\end{array}%
.
\]

Our methodology has been spelt out. Use the low degree coordinates of the
Signature of the normalised financial market data $\gamma $ as features $%
\phi _{i}\left( \gamma \right) $, use least squares on the learning set to
approximately reproduce $f$%
\[
f\left( \gamma \right) \approx \sum_{i}\lambda _{i}\phi _{i}\left( \gamma
\right)
\]%
and then test it on the backtesting set. To summarise the methodology:

\begin{enumerate}
\item We used futures data normalised to remove volume and volatility
information.

\item We used linear regression based pair-wise separation to find the best
fit linear function to the learning pairs that assign 0 to one case and 1 to
the other. (There are other well known methods that might be better.)

\begin{enumerate}
\item We used robust and automated repeated sampling methods of LASSO type (least absolute
shrinkage and selection operator) based on constrained $L^1$ optimisation to
achieve shrinkage of the linear functional onto an expression involving only
a few terms of the Signatures.
\end{enumerate}

\item and we used simple statistical indicators to indicate the
discrimination that the learnt function provided on the learning data and
then on the backtesting data. The tests were:

\begin{enumerate}
\item Kolmogorov-Smirnov distance of distributions of score values

\item receiver operating characteristic (ROC) curve, area under ROC curve

\item ratio of correct classification.
\end{enumerate}
\end{enumerate}

We did consider the full range of half hour time intervals. The other time
intervals were not readily distinguishable from each other but were easily
distinguishable from both of these two time intervals using the methodology
mapped out here. It seems likely that the differences identified here were
due to distinctive features of the market associated with the opening and
closing of the open outcry market.

\newpage

%\begin{figure}[H][ht] 
%\centering
%%left bottom right top
%\subfigure[\textbf{Learning
%set}: Estimated densities of the regressed values, K-S distance: $0.9$,
%correct classification: $95\%$]{
%\includegraphics[trim = 10mm 5mm 154mm 15mm, clip,width =
%0.47\textwidth]{CLN_NOV_plot_CDF_2011_2012_2013_930vs1400_PDF}
%}
%\subfigure[\textbf{Out of
%sample}: Estimated densities of the regressed values, K-S distance:
%$0.91$, correct classification: $95\%$]{
%\includegraphics[trim = 154mm 5mm 10mm 15mm, clip,width =
%0.47\textwidth]{CLN_NOV_plot_CDF_2011_2012_2013_930vs1400_PDF}
%}
%\subfigure[\textbf{ROC curve.} Area under ROC -- learning set: 0.986, out of
%sample: 0.984 ]{
%\includegraphics[trim = 10mm 5mm 10mm 10mm, clip,width =
%0.8\textwidth]{CLN_NOV_plot_ROC_2011_2012_2013_930vs1400}
%}
%\caption{14:00-14:30 EST versus 9:30-10:00 EST}
%\label{fig:140vs930}
%\end{figure}[H]

\begin{figure}[H]
\centering
%left bottom right top
\subfigure[\textbf{Learning
set}: Estimated densities of the regressed values, K-S distance: $0.8$, correct classification: $90\%$]{
\includegraphics[trim = 10mm 5mm 154mm 15mm, clip,width =
0.47\textwidth]{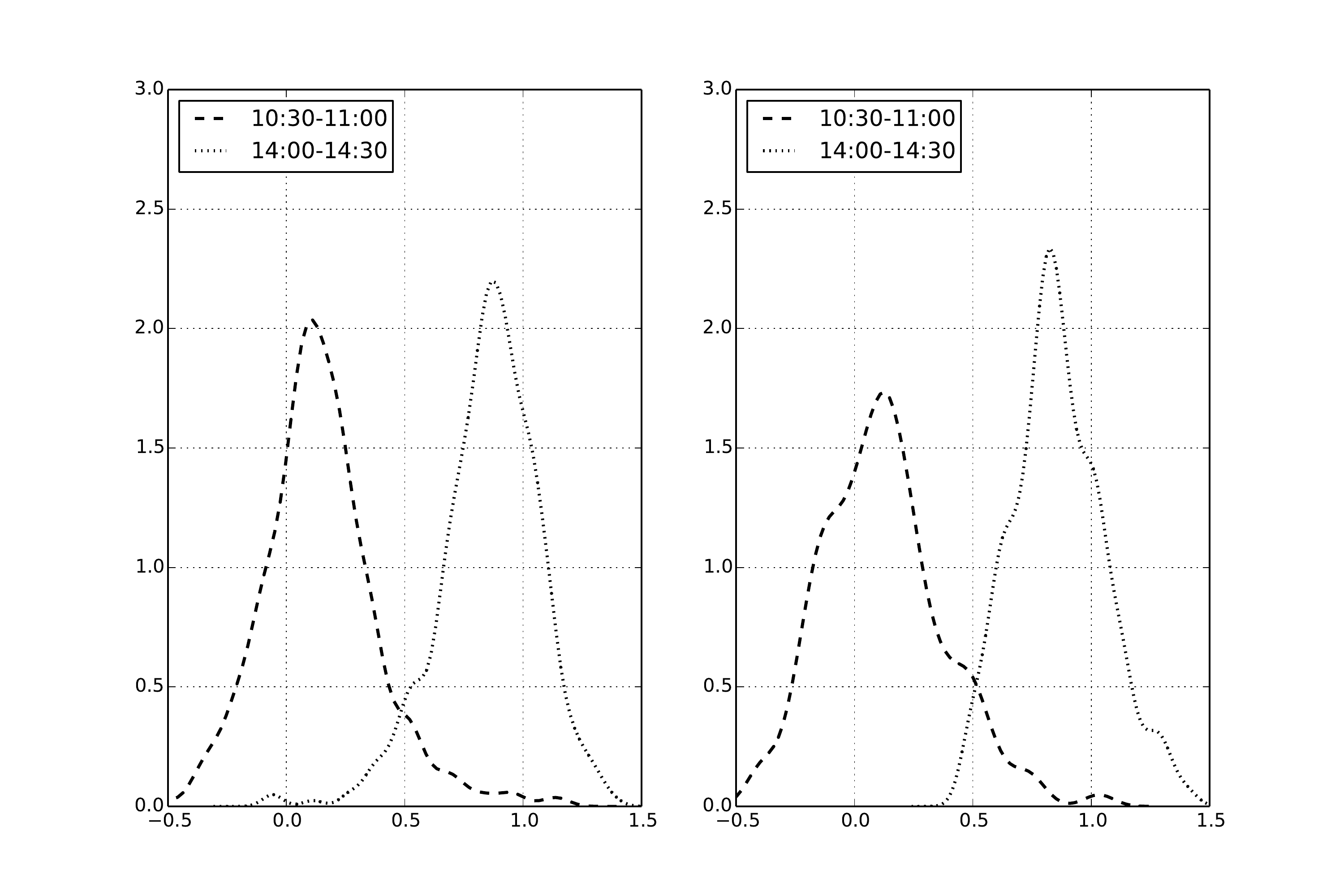}
} 
\subfigure[\textbf{Out of
sample}: Estimated densities of the regressed values, K-S distance:
$0.84$, correct classification: $89\%$]{
\includegraphics[trim = 154mm 5mm 10mm 15mm, clip,width =
0.47\textwidth]{CLN_NOV_plot_CDF_2011_2012_2013_1030vs1400_PDF}
} 
\subfigure[\textbf{ROC curve.} Area under ROC -- learning set: 0.976, out of
sample: 0.986 ]{
\includegraphics[trim = 10mm 5mm 10mm 10mm, clip,width =
0.8\textwidth]{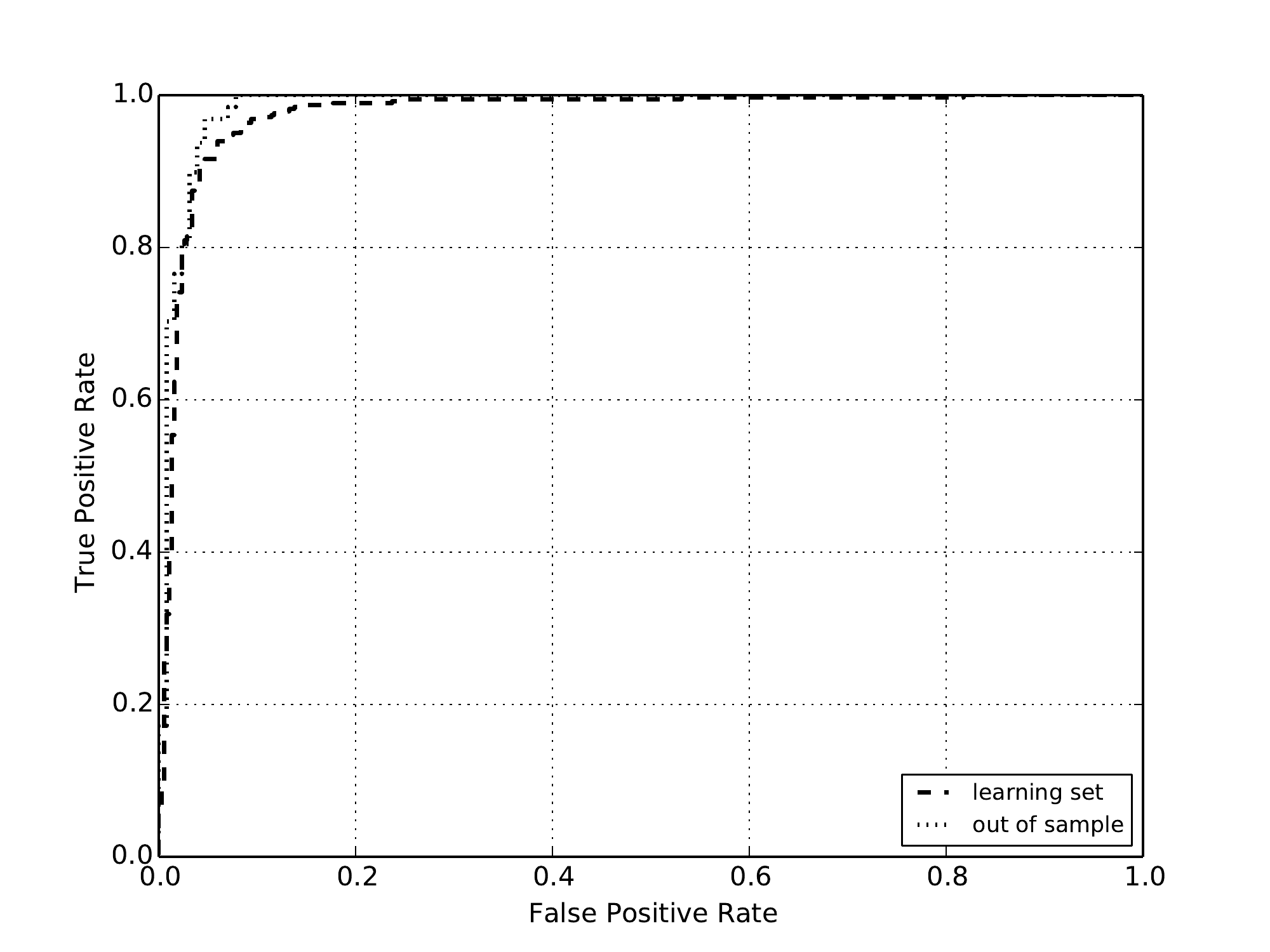}
}
\caption{14:00-14:30 EST versus 10:30-11:00 EST}
\label{fig:140vs1030}
\end{figure}

\begin{figure}[H]
\centering
\subfigure{
\includegraphics[trim = 5mm 5mm 5mm 5mm, clip,
width=0.45\textwidth]{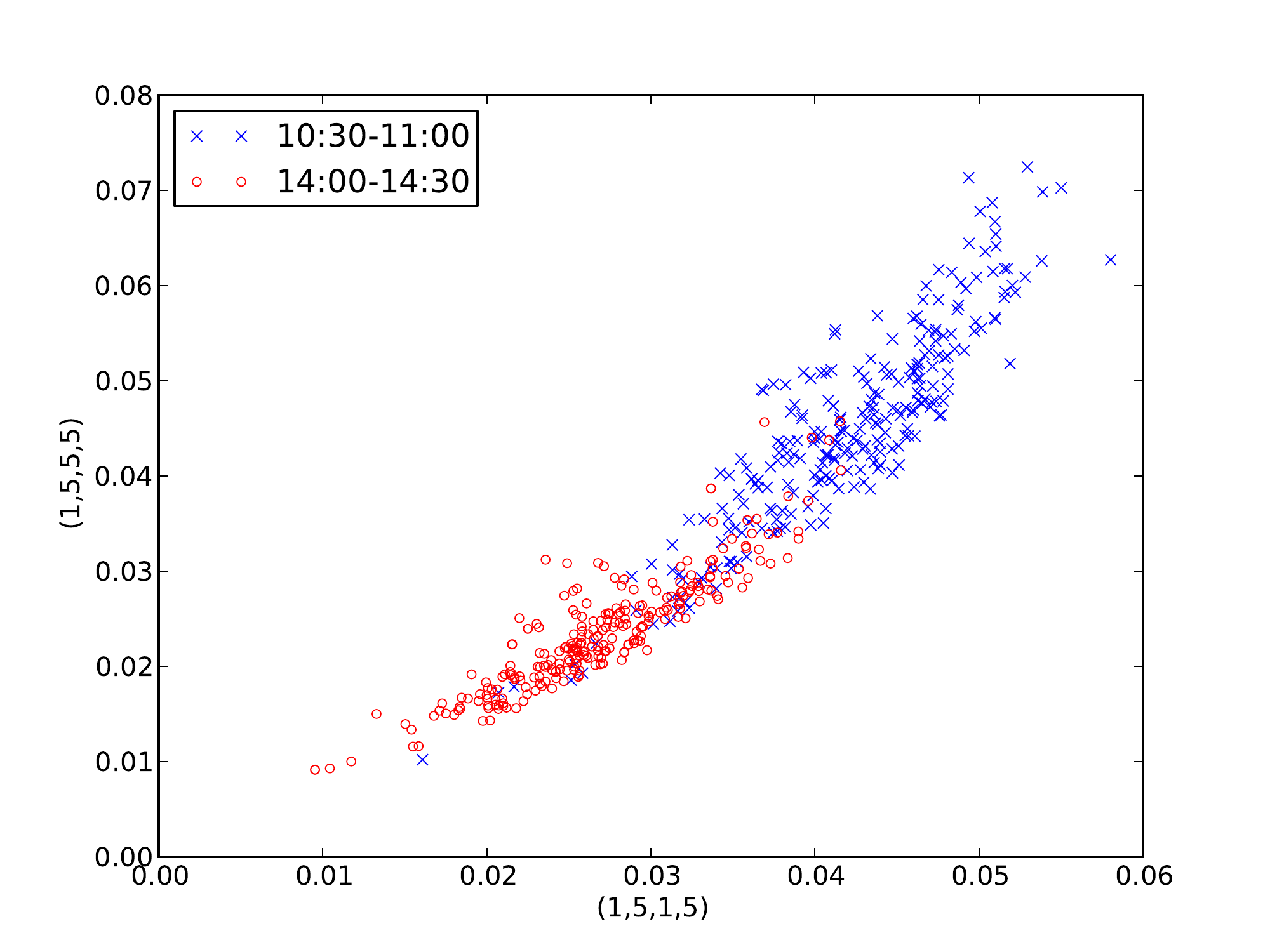}
} 
\subfigure{
\includegraphics[trim = 5mm 5mm 5mm 5mm, clip,
width=0.45\textwidth]{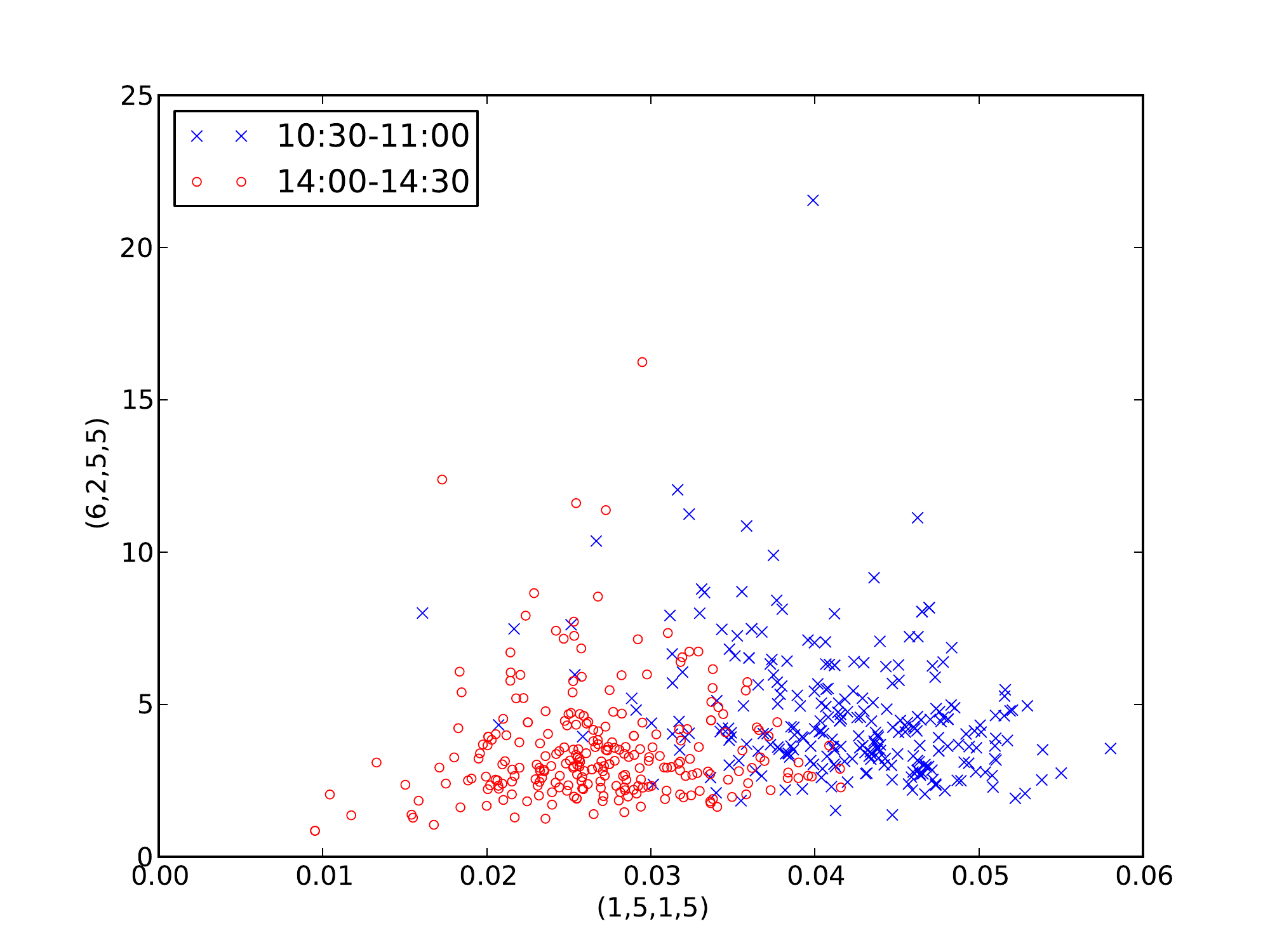}
} 
\subfigure{
\includegraphics[trim = 5mm 5mm 5mm 5mm, clip,
width=0.45\textwidth]{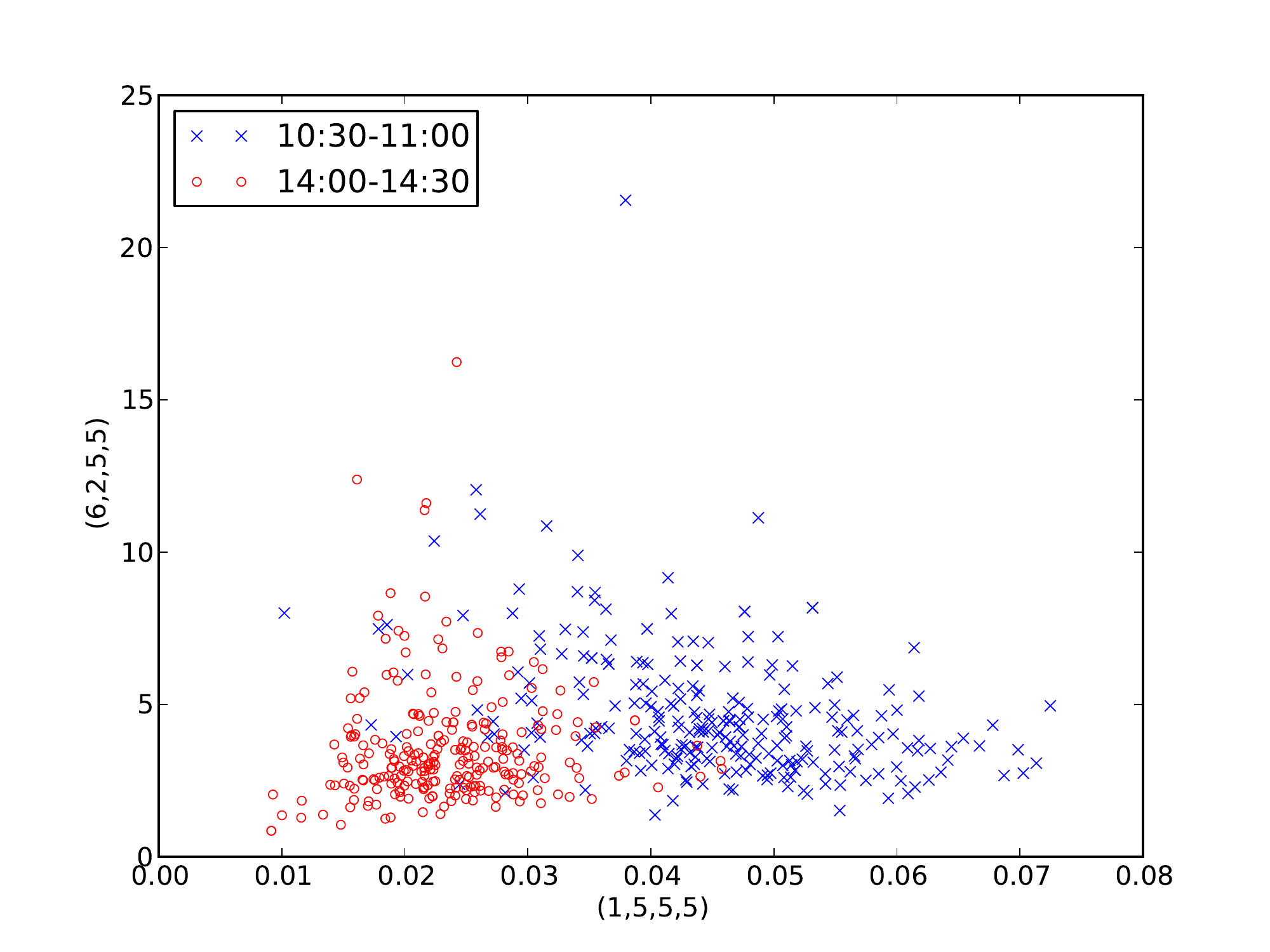}
} 
\subfigure{
\includegraphics[trim = 5mm 5mm 5mm 5mm, clip,
width=0.45\textwidth]{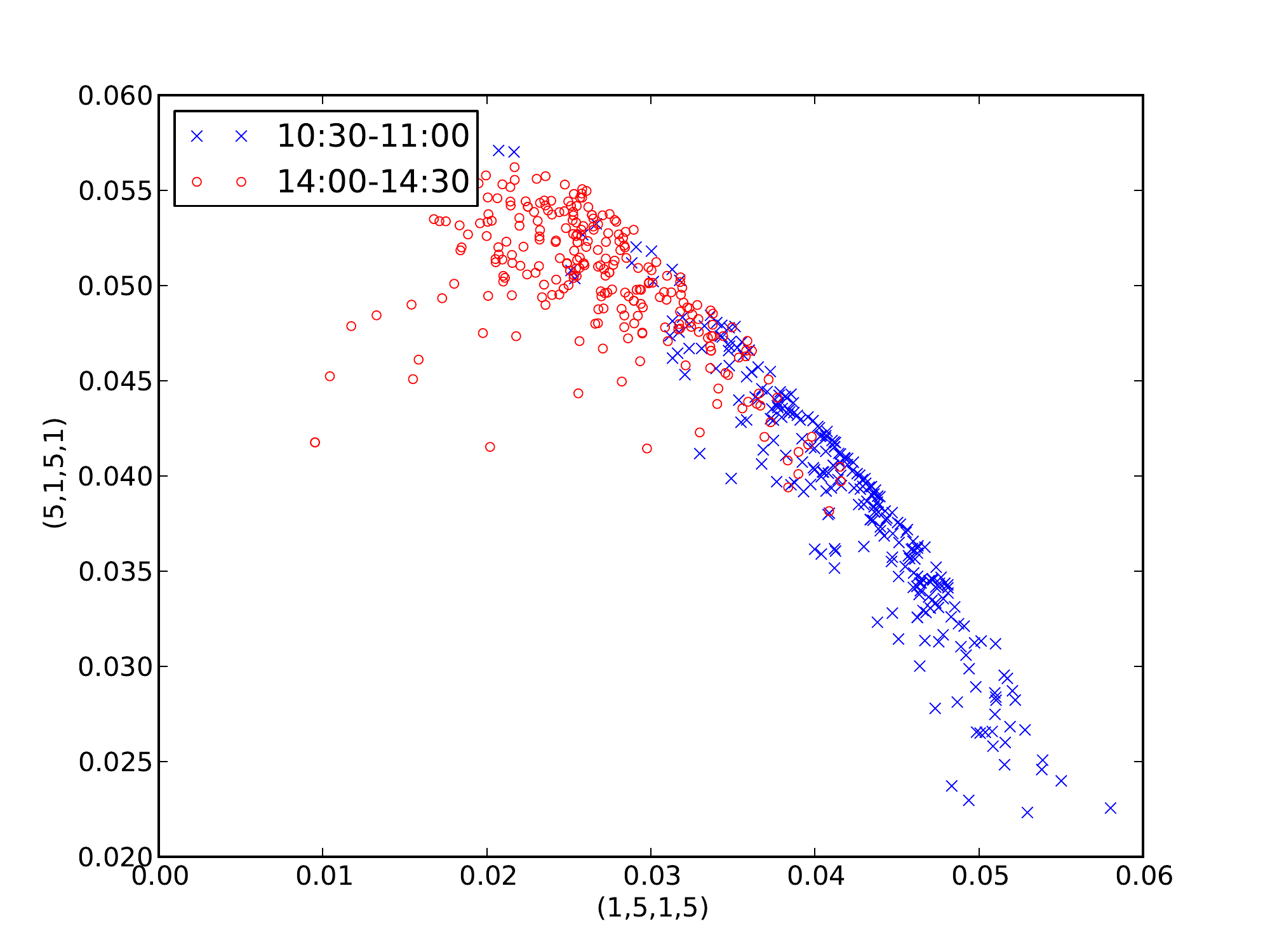}
}
\caption{Visualisation: two dimensional projections of the 4th order signature onto coefficients selected as significant by Lasso shrinkage. The selected features allow clear visual separation of the time buckets.}
\label{fig:visualisation}
\end{figure}

%*********************************************************
%\newpage
%
%\subsection{classification of time-buckets from standardised data - results 4%
%}
%
%\vspace{16pt} \textbf{commodity future, front month:} 12:00-12:30 vs
%12:30-13:00
%
%%left, bottom, right and top
%\begin{figure}[H][H]
%\centering
%\subfigure{
%\includegraphics[trim = 5mm 5mm 5mm 5mm, clip,
%width=0.45\textwidth]{plotColor_1200_1230_2012_0_1}
%} 
%\subfigure{
%\includegraphics[trim = 5mm 5mm 5mm 5mm, clip,
%width=0.45\textwidth]{plotColor_1200_1230_2012_0_3}
%} 
%\subfigure{
%\includegraphics[trim = 5mm 5mm 5mm 5mm, clip,
%width=0.45\textwidth]{plotColor_1200_1230_2012_1_3}
%} 
%\subfigure{
%\includegraphics[trim = 5mm 5mm 5mm 5mm, clip,
%width=0.45\textwidth]{plotColor_1200_1230_2012_0_2}
%}
%\end{figure}[H]

\section{Linear regression onto a law on paths}

In the previous section we looked at using the linearising nature of the
Signature as a pratical tool for learning functions. In this final section
we want to remain in the world of data and applications but make a more
theoretical remark. Classic nonlinear regression is usually stated with a
statistical element. One common formulation of linear regression has that a
stationary sequence of random data pairs that are modeled by 
\[
y_{i}=f\left( x_{i}\right) +\varepsilon _{i}
\]%
where $\varepsilon _{i}$ is random and has conditional mean zero. The goal is to
determine the linear functional $f$ with measurable confidence.

There are many situations where it is the case that one has a random but
stationary sequence $\left( \gamma ,\tau \right) $ of stream pairs, and one
would like to learn, approximately, the law of $\tau $ conditional on $%
\gamma $. Suppose that we reformulate this problem in terms of Signatures and
expected Signatures (or better: charateristic functions) recalling that expected Signatures  etc. characterise laws.

\begin{problem}
Given a random but stationary sequence $\left( \gamma ,\tau \right) $ of
stream pairs find the function $\Phi :S\left( \gamma \right) \rightarrow 
\mathbb{E}\left( S\left( \tau \right) |S\left( \gamma \right) \right) .$
\end{problem}

Then putting $Y_{i}=S\left( \tau _{i}\right) $ and $X_{i}=S\left( \gamma
_{i}\right) $ we see that%
\[
Y_{i}=\Phi \left( X_{i}\right) +\varepsilon _{i}
\]%
where $\varepsilon _{i}$ is random and has mean zero. If the measure is
reasonably localised and smooth then we can well approximate $\Phi $ by a
polynomial; and using th elinearising nature of the tensor algebra to a linear function $\phi $ of the Signature. 
In other
words the apparently difficult problem of understanding conditional laws of
paths becomes (at least locally) a problem of linear regression 
\[
Y_{i}=\Phi \left( X_{i}\right) +\varepsilon _{i}
\]%
whch is infinite dimensional but which has well defined low dimensional
approximations \cite{levin2013learning}.

\bibliographystyle{amsplain}
\bibliography{citations}

\end{document}